\documentclass[twoside, 12pt]{amsart}
\usepackage{latexsym}
\usepackage{amssymb,amsmath,amsopn}
\usepackage[dvips]{graphicx}   
\usepackage{color,epsfig}      

\newtheorem{thm}{Theorem}

\newtheorem{thm2}{Theorem}[section]
\newtheorem{lem}{Lemma}[section]

\newtheorem{cor}[lem]{Corollary}

\newtheorem{prop}[lem]{Proposition}

\theoremstyle{definition}
\newtheorem{defn}[lem]{Definition}

\newtheorem{rem}[lem]{Remark}

\renewcommand{\Re}{\mathbb R}
\newcommand{\Ren}{\Re^n}

\newcommand{\BB}{\mathbf B}

\renewcommand{\S}{\mathbb{S}}

\newcommand{\F}{\mathcal{F}}
\newcommand{\B}{\BB}
\newcommand{\remark}[1]{}

\DeclareMathOperator{\bd}{bd}

\DeclareMathOperator{\dist}{dist}
\DeclareMathOperator{\conv}{conv}

\DeclareMathOperator{\diam}{diam}

\DeclareMathOperator{\relint}{relint}

\newcommand\fact[2][]{
\ifx&#1& 
 \refstepcounter{equation}
\fi
\begin{minipage}{0.002\textwidth}
\ifx&#1& 
 (\theequation)
 \else 
 #1
\fi
\end{minipage}
\hspace{0.08\textwidth}
\begin{minipage}{0.85\textwidth}
\emph{#2}
\end{minipage}
}

\begin{document}
\title[On multiple Borsuk numbers]{On multiple Borsuk numbers in normed spaces}
\author[Z. L\'angi and M. Nasz\'odi]{Zsolt L\'angi and M\'arton Nasz\'odi}
\address{Zsolt L\'angi, Dept.\ of Geometry, Budapest University of
Technology and Economics, Budapest, Egry J\'ozsef u. 1., Hungary, 1111}
\email{zlangi@math.bme.hu}
\address{
M\'arton Nasz\'odi,
Dept. of Geometry,
Lorand E\"otv\"os University,
P\'azm\'any P\'eter S\'et\'any 1/C
Budapest, Hungary 1117
}
\email{marton.naszodi@math.elte.hu}
\keywords{Borsuk's problem, diameter, unique completion, covering, bodies of
constant width, complete, multiple chromatic number.}
\subjclass{52C17, 05C15, 52A21}
\thanks{The authors gratefully acknowledge the support of the J\'anos Bolyai Research Scholarship of the Hungarian Academy of Sciences, and
the Hung. Nat. Sci. Found. (OTKA) grants: K72537 and PD104744.}

\begin{abstract}
Hujter and L\'angi defined the \emph{$k$-fold Borsuk number} of a set $S$ in
Euclidean $n$-space of diameter $d > 0$ as the smallest cardinality of a family
$\F$ of subsets of $S$, of diameters strictly less than $d$, such that every
point of $S$ belongs to at least $k$ members of $\F$.

We investigate whether a $k$-fold Borsuk covering of a set $S$ in a finite dimensional real normed
space can be extended to a completion of $S$. Furthermore, we determine the
$k$-fold Borsuk number of sets in not angled normed planes, and give a partial characterization
for sets in angled planes.
\end{abstract}

\maketitle

\section{Introduction}

In 1933, Borsuk \cite{B33} posed the problem whether any set $S$ of diameter $d
> 0$ in Euclidean $n$-space $\Re^n$ is the union of $n+1$ sets of diameters
less than $d$. A proof of the affirmative answer for $n=2$ appeared in
\cite{B33}, and for $n=3$ in \cite{E55} (for finite $S$, see \cite{G57_2},
\cite{H57}). Sixty
years after the problem appeared, Kahn and Kalai \cite{K93} proved that for
large values of $n$ the answer is negative. For surveys on Borsuk's problem,
see \cite{BMS97, Ra08}.

Boltyanski \cite{B70} gave a characterization of bounded sets according to
their \emph{Borsuk number} (that is, the least number of smaller diameter
pieces that they can be partitioned into) in the Euclidean plane: \emph{Let
$\emptyset\neq S\subset\Re^2$ be a bounded set that is not a singleton. Then the
Borsuk number of $S$ is 3 if $S$ has a unique completion (see
Definition~\ref{defn:completion}) and 2 otherwise.}

Gr\"unbaum \cite{G57} was the first to consider the Borsuk numbers of
sets with respect to a metric distinct from the Euclidean, and determined the
Borsuk numbers of sets in the plane equipped with the $\ell_{\infty}$ norm. The
problem was solved for arbitrary normed planes in \cite{BS77}:

\begin{thm2}[Boltyanski-Soltan]\label{thm:normedBorsuk}
  Let $S$ be a compact set in the normed plane with unit ball $\B$.
  Then the Borsuk number of $S$ is
  \begin{itemize}
  \item $a(S) = 4$ if, and only if, $\B$ and $S$ are homothetic parallelograms;
  \item $a(S) = 3$ if, and only if, $a(S) \neq 4$, there is a unique
  completion $C$ of $S$ with respect to $\B$, and $S$ satisfies the \emph{supporting line property}\label{supplineprop}: for any pair of parallel supporting
lines of $C$, $S$ has a point on at least one;
  \item $a(S) = 2$ otherwise.
  \end{itemize}
\end{thm2}

As a generalization of Borsuk's problem, Hujter and L\'angi \cite{HL13}
defined the \emph{$k$-fold Borsuk number}, $a_k(S)$, of a set $S$ of
diameter $d > 0$ as the smallest cardinality of a family
$\F$ of subsets of $S$, of diameter strictly less than $d$, such that every
point of $S$ belongs to at least $k$ members of $\F$. Among other results, they
determined the $k$-fold Borsuk numbers of any set in the Euclidean plane.

Motivated by Boltyanski's result, we investigate whether a ($k$-fold) Borsuk
covering of a set $S$ can be extended to a completion of $S$.
Theorem~\ref{thm:normedextensionSC} states that such an extension is possible
in certain \emph{Minkowski spaces} (ie. finite dimensional real normed spaces)
provided that $S$ has a unique completion. The
class of these Minkowski spaces include Euclidean $n$-space for all $n$. This
result has been known in the Euclidean plane \cite{B70} but is new in higher
dimensional Euclidean spaces. In Theorem~\ref{thm:completionNotAngled}, we
extend this result to not angled Minkowski planes (see
Definition~\ref{defn:notangled}).

In Theorems~\ref{thm:notReuleaux}, \ref{thm:Reuleaux} and
\ref{thm:angled}, we find the $k$-fold Borsuk numbers of sets in not angled normed planes,
and of sets that cannot be completed uniquely to a Reuleaux polygon in angled planes.

\section{Definitions and notations}

We denote the closed unit ball centered at a point $x\in\Ren$ of a Minkowski
space by $\B(x)$, and its boundary, the unit sphere by $\S(x)$. For a set $A$,
the intersection of unit balls centered at the points of $A$ is denoted as
\[
 \B A=\bigcap_{x\in A} \B(x).
\]

\begin{defn}\label{defn:completion}
A bounded set $C$ in an $n$-dimensional Minkowski space is \emph{complete}, if
no set
of the same diameter properly contains $C$. (Note that a complete set is
clearly compact and convex.)
A set $S$ is a \emph{set of unique
completion} if there is a unique complete set $C$ containing $S$ of the same
diameter as $S$.
\end{defn}

\begin{prop}\label{prop:udm}
Let $S$ be a set of unit diameter in an $n$-dimensional Minkowski space. Then
 \begin{itemize}
  \item $S$ is complete if, and only if $S=\B S$,
  \item $S$ is a set of unique completion if, and only if, $\B S=\B^2 S$ ie.
$\B
S$ is complete.
 \end{itemize}
\end{prop}
The first statement is due to Eggleston \cite{E65}, where it is called the
\emph{spherical intersection property}, the second is due to Moreno (Corollary~3
in \cite{M11}). Note that in the second case the completion of $S$ is $\B S$.

We define the \emph{distance of a set} $A$ of a Minkowski space and a
point $x$ as $d_{\B}(x,A)=\inf\{d_{\B}(x,a) : a\in A\}$,
where $d_{\B}(x,a)$ is the distance of the points $a$ and $x$ in the normed space with unit ball $\B$.

\section{Extending a Borsuk covering in certain Minkowski
spaces}\label{sec:extension}

Our goal is to extend a Borsuk covering of a closed set $S$ of unique
completion in a Minkowski space to its unique completion $\B S$. In
general, a Borsuk covering of a compact set may not extend to any of its
completions: consider a pair of points which in Euclidean space have many
completions, all of whose Borsuk number is above two.

\subsection{Extension of a Borsuk covering in certain Minkowski spaces}
We define the following ``Lens Cutting Condition'' which holds in certain
Minkowski spaces:

\medskip

\fact[(LCC)]{For any two distinct points $u$ and $v$ in $\Ren$ with
$d_{\B}(u,v)\leq
1$ and
$x\in\S(u)\cap\S(v)$ and $\varepsilon>0$, there is a $w\in\Ren$ such that
$x\notin \B(w)$ but $\B(w)\supset\B(u)\cap\B(w)\setminus\varepsilon\B(x)$.
}

\begin{rem}
 It is not hard to see that (LCC) holds in all Euclidean spaces.
\end{rem}

\begin{thm}\label{thm:normedextensionSC}
 If (LCC) holds in a Minkowski space then any $k$-fold Borsuk covering of a
closed set of unique completion extends to a $k$-fold Borsuk covering of its
completion.
\end{thm}

\begin{proof}
We prove the Theorem for $k=1$, the general case follows from the same
argument. Let $S=Q_1\cup\ldots\cup Q_k$ be a Borsuk covering of a closed set
$S$ of unique completion by closed sets of diameter at most $r<1$.
Note that a Borsuk covering of the boundary of a set may be extended to the set in a straightforward way (cf. also Remark~\ref{rem:boundary}).
Thus, we will define sets $Q^{\prime}_1\cup\ldots\cup Q^{\prime}_m=\bd \B S$ that form a Borsuk
covering of the boundary of the completion $\B S$ of $S$.

For all $i$, $Q^{\prime}_i$ will contain $Q_i\cap \bd \B S$ and some more points of $\bd \B
S$. For an $x\in(\bd \B S)$ we take the index $i$ such that $d(x,Q_i)$ is
minimal (if it is not unique, we take all such $i$), and include $x$ into
$Q^{\prime}_i$. Clearly, $Q^{\prime}_i$ is closed.

Note that for any $x\in \B S\setminus S$ we have that

\fact[(*)]{there are no two distinct points $u, v\in \B S$ with
$d_{\B}(x,u)=d_{\B}(x,v)=1$.
}

Suppose the contrary. Then $S\subseteq \B^2 S\subseteq \B(u)\cap \B(v)$. On
the other hand, $\B^2 S$ is the intersection of all unit balls that contain $S$,
and hence by (LCC), $\B^2 S\subseteq (\B(u)\cap \B(v))\setminus \{x\}$,
contradicting $x\in \B S=\B^2 S$.

The family of the sets $q_i'$ is a Borsuk partition of $\bd \B S$. Indeed, let $x,y\in Q^{\prime}_i$. If
$x$ or $y$ is in $S$ then clearly, $d(x,y)<1$. If both are in
$Q_i^{\prime}\setminus S$ then, by (*), $d(x,y)<1$.
\end{proof}

\subsection{Extension of a Borsuk covering in certain Minkowski planes}
It is not difficult to see that a strictly convex normed plane (that is, when
the unit disk $\B$ is strictly convex) satisfies (LCC), and thus has the
extension property of Theorem~\ref{thm:normedextensionSC}. Next, we consider a
class of Minkowski planes that is wider than the class of strictly convex
planes, and where (LCC) does not hold, but the extension property still does.
The following definition is from \cite{BMS97} (cf. also \cite{BS77}).

\begin{defn}\label{defn:notangled}
A normed plane with unit ball $\B$ is \emph{angled}, if for some non-collinear
points $a,b,c$, we have $[a,b] \cup [b,c] \subset \S$.
\end{defn}

\begin{thm}\label{thm:completionNotAngled}
Let $S$ be a set of unique completion in a not angled normed plane, and let $C$
be the completion of $S$. Then any $k$-fold Borsuk covering $\F$ of $S$ can be
extended to a $k$-fold Borsuk covering of $C$.
\end{thm}

From this point on throughout this section, we assume that the Minkowski plane
we work with is not angled.

The following monotonicity lemma appeared in \cite{L04}.

\begin{lem}[Lassak]\label{lem:monotonicity}
Let $t \mapsto p(t)$ (with $t \in [0,1]$) be a simple, closed, continuous
curve, defining the boundary of a complete body of diameter one in a Minkowski
plane. Let $p=p(0)$, and let $t_1$ and $t_2$ be the smallest and the largest
values of $t$ such that $\dist_\B(p,p(t))=2$. Then the function $t \mapsto
\dist_\B(p,p(t))$ is
\begin{itemize}
\item strictly increasing on the interval $[0,t_1]$,
\item equal to one on $[t_1,t_2]$, and
\item strictly decreasing on $[t_2,1]$.
\end{itemize}
\end{lem}

\begin{cor}
Let $C$ be a complete body of diameter one in a Minkowski plane. Then, for any
$p \in \bd C$ we have the following.
\begin{itemize}
\item The set of points of $C$ at unit distance from $p$ is a connected arc of
$\S(p) \cap \bd C$.
\item If $||q-p||_{\B} = ||r-p||_{\B}$ for some $q,r \in \bd C$, then the arc
of $\bd
C$, connecting $q$ and $r$ and not containing $p$, belongs to the
circle $\S(p)$.
\end{itemize}
\end{cor}

\begin{lem}\label{lem:disjointdiam}
If $C$ is a complete body in a Minkowski plane,
and $[a,b]$, $[c,d]$ are two disjoint diameters of $C$ such that $a,b,c,d$ are
in counterclockwise order in $\bd C$, then
$[a,d], [b,c] \subset \bd C$ and they are parallel.
\end{lem}

\begin{proof}
Consider the quadrangle $Q=\conv \{ a,b,c,d\}$.
Observe that as $[a,b]$ and $[c,d]$ are diameters of $C$, neither $C$ nor $Q
\subseteq C$ contains a translate of neither $[a,b]$ nor $[c,d]$
in its interior. Thus, $[a,c]$ and $[b,d]$ are parallel, and they belong to $\bd
C$.
\end{proof}

Lemma~\ref{lem:supportinglineproperty} is a straightforward consequence of
Theorems 33.7 and 33.9 of \cite{BMS97}.

\begin{lem}\label{lem:supportinglineproperty}
Let $S$ be a compact set of unique completion in a not angled normed plane, and
let $C$ be its
completion.
Then, for any parallel supporting lines $L$ and $L'$ of $C$, $L$ or $L'$
contains a point of $S$. In other words, $S$ satisfies the supporting line property (see page~\pageref{supplineprop}).
\end{lem}

\begin{lem}\label{lem:completionarc}
Let $S$ be a compact set of diameter one and of unique completion, $C$.
Then, for any point $x \in (\bd C) \setminus S$, there is an open circle arc of
radius one, containing $p$ and being
contained in $\bd C$, such that its endpoints and its center belong to $S$.
\end{lem}

\begin{proof}
Let $x \in (\bd C) \setminus S$.
Then, since $\B S = C$, there is a point $p \in S$ such that $x \in \S(p)$.
Clearly, $[p,x]$ is a diameter of $C$, and thus, $p \in \bd C$.
Let $L$ and $L'$ be a pair of parallel supporting lines of $C$ such that $x \in
L$ and $p \in L'$.
For simplicity, we imagine these lines as vertical such that $L$ is to the left
of $L'$.
Let $[a,b] = C \cap L$ and $[c,d] = C \cap L'$, and note that these segments
might be degenerate.
Without loss of generality, we assume that $a$, $b$, $c$ and $d$ are in this
counterclockwise order in $\bd C$.

First, we show that at least one of $a$ and $c$ belongs to $S$.
Indeed, consider a sequence of supporting lines $L_m$ of $C$, with positive
slopes, such that the limit of $L_m \cap C$ is  $\{ a \}$.
For any $m$, let $L'_m$ be the supporting line of $C$, parallel to and different
from $L_m$.
Clearly, the limit of $L'_m \cap C$ is $\{ c \}$.
Now, by Lemma~\ref{lem:supportinglineproperty}, we have that for any $m$, $L_m$
or $L'_m$ contains a point of $S$.
Thus, the observation follows from the compactness of $S$.
We may show similarly that at least one of $b$ and $d$ belongs to $S$.

Now we prove the assertion.
If both $[a,x]$ and $[x,b]$ contain a point of $S$, then we may observe that
$[a,b] \subset \S(p)$, and thus, our lemma follows.
Assume that exactly one of these segments, say $[x,b]$, contains a point of $S$.
Then $a \notin S$, and thus, $c \in S$.
Let $G$ be the arc of $(\bd C) \cap \S(c)$, starting at $x$ and
above the line connecting $x$ and $c$.
If $G$ does not contain a point of $S$, then for some point $c' \in L' \setminus
[c,d]$, we have $S \subset \B(c')$; or in other words, $c' \in \B S = C$; a
contradiction.
Thus, $G$ contains a point of $S$, which yields the assertion.

We are left with the case that $[a,b] \cap S = \emptyset$, which, in particular,
implies that $c,d \in S$.
Note that if $c \neq d$, then, by Lemma~\ref{lem:disjointdiam}, for any $y \in
\relint [c,d]$, we have $C \cap \S(y) = [a,b]$.
Thus, moving $y$ slightly to the right, we can find a point $y'$ such that $S
\subset \B(y')$, but $[a,b] \cap \B(y') = \emptyset$.
This yields that $C \not\subset \B(y')$, or in other words that $y' \notin \B C
= C$, contradicting $y' \in \B S=C$.
Thus, we obtain that $c=d$. In this case, similarly like in the previous
paragraph, one can show that both arcs of $(\bd C) \cap \B(c)$,
starting at $x$, contain a point of $S$, and the assertion readily follows.
\end{proof}

\begin{proof}[Proof of Theorem~\ref{thm:completionNotAngled}]
Note that it suffices to extend $\F$ to a $k$-fold Borsuk covering of $\bd C$.

Let $\F= \{ Q_1, Q_2, \ldots, Q_m\}$ be a $k$-fold Borsuk covering of $S$.
Without loss of generality, we may assume that $S$ is compact.
Let $\varepsilon$ be chosen in such a way that the diameter of every member of
$\F$ is at most $1-3\varepsilon$.
Now, for every $i$, we set $Q^*_i = Q_i + \varepsilon B$, and observe that $\F^*
= \{ Q^*_1, Q^*_2, \ldots, Q^*_m\}$
is still a $k$-fold Borsuk covering of $S$.

Consider the connected components of $\bd C \setminus S$. By
Lemma~\ref{lem:completionarc},
they are open circle arcs of unit radius, with their centers contained in $S$.
Note that $\F^*$ is a $k$-fold covering of any such arc not longer than
$2\varepsilon$.
Since $\bd C$ has a bounded length, there are only finitely many arcs that are
not covered $k$-fold by $\F^*$.
Thus, by induction, it suffices to prove that $\F^*$ can be extended to cover
$k$-fold at least one such arc.

Consider an arc $G$ that is not covered by $\F^*$ $k$-fold. Let $p \in S$ denote
the center, and $q,r \in S$ denote the endpoints of $G$.
If, for every $x \in G$, $p$ is the only point of $C$ at unit distance from $x$,
then we can apply the argument in
the proof of Theorem~\ref{thm:normedextensionSC}.
Thus, assume that for some $x \in G$ and $p' \in C$ with $p \neq p'$, we have $x
\in \S(p')$, where
without loss of generality, we may assume that, say, $[p,r]$ and $[p',x]$ are
disjoint.
Note that since $[p,r]$ and $[p',x]$ are diameters of $C$, we have that $[p,p']$
and $[x,r]$ are parallel, and are contained in $\bd C$.

Let $L$ and $L'$ be the line containing $[r,x]$ and $[p,p']$, respectively.
Observe that the points diametrically opposite to any point in the relative
interior of $L \cap C$ are the points of $L' \cap C$.
Let $y$ be the endpoint of $L \cap \bd C$ closer to $x$ than to $r$.
If $q \in L$, then we may add the segment $[q,x]$ to any $Q_i^*$
containing $q$, and $[x,r]$ to any $Q_i^*$ containing $r$. Thus, we may assume
that $y$ is a point of $G$.

Consider the case that the points diametrically opposite to $y$ are only the points
of $L' \cap C$. Then we may add the segment $[y,r]$ to any $Q_i^*$ containing $r$.
On the other hand, note that if some $u \in \bd C$ is diametrically opposite to any
point of the arc between $y$ and $q$, then it is diametrically opposite to $q$ as well.
Thus, we may add this arc to any $Q_i^*$ containing $q$.

\begin{figure}[ht]
\begin{center}
\includegraphics[width=0.5\textwidth]{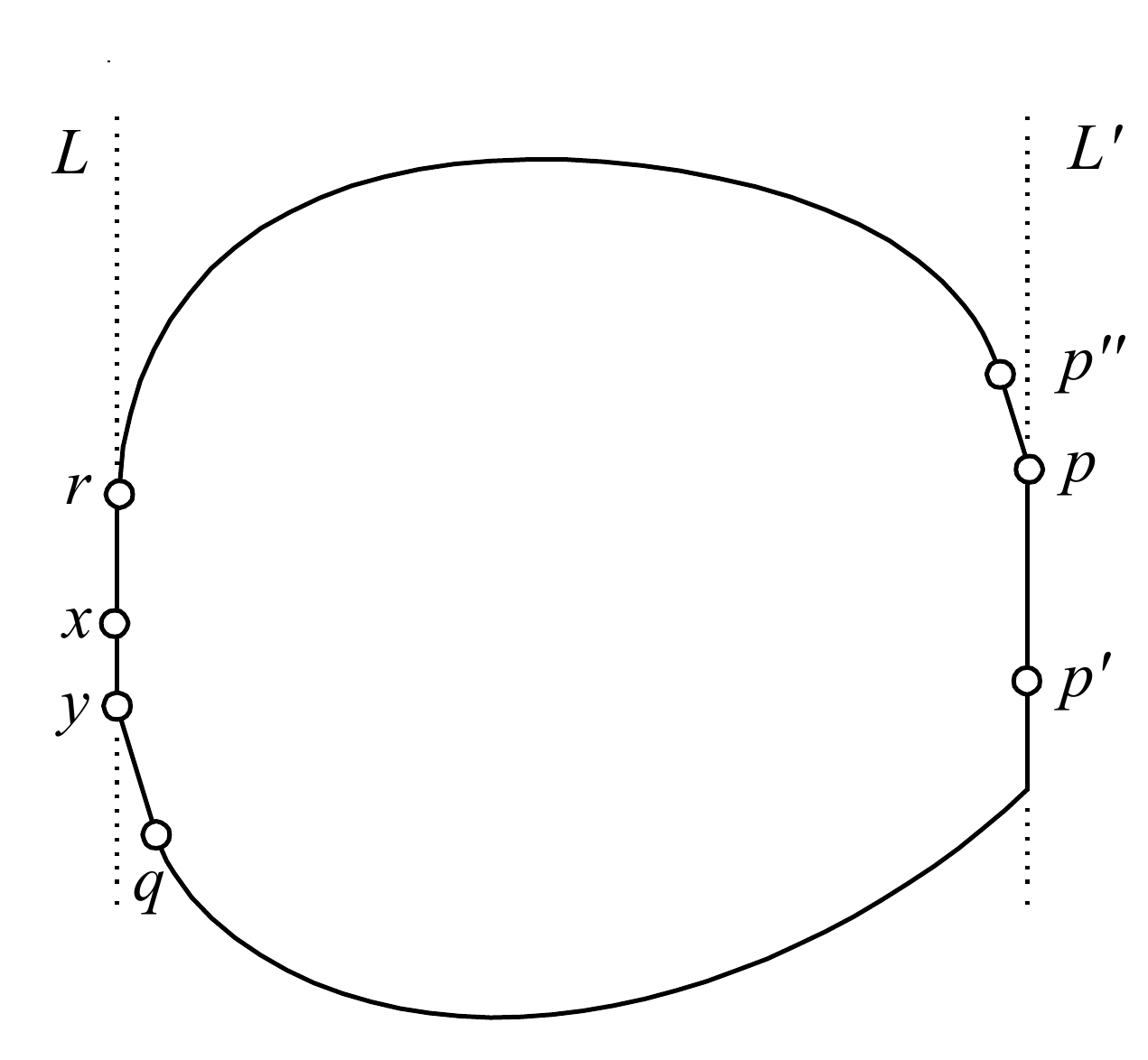}
\caption[]{An illustration for the proof of Theorem~\ref{thm:completionNotAngled}}
\label{fig:uniqueextension}
\end{center}
\end{figure}

Finally, assume that there is some point $p'' \notin L' \cap C$ that is
diametrically opposite to $y$ (cf. Figure~\ref{fig:uniqueextension}).
Then, clearly, the points $p',p,p'',y$ are in this cyclic order in $\bd C$,
and $[y,p'']$ and $[q,p]$ are disjoint diameters of $C$, which yields, by Lemma~\ref{lem:disjointdiam},
that $[p,p'']$ and $[q,y]$ are parallel, and both are contained in $\bd C$.
Thus, $\bd C$, and also $\S(p)$, contains an angle, which
contradicts the conditions of the theorem.
\end{proof}

\begin{cor}\label{cor:udmBorsukcovering}
Let $S$ be a set of unique completion in a not angled normed
plane, and let $C$ be the completion of $S$. Then for any value of $k$, $a_k(S)
= a_k(C)$.
\end{cor}

\section{The multiple Borsuk numbers of sets in a not angled normed plane}

We start with three observations, which, for sets in a Euclidean space, appeared
as Remarks 1--3 in \cite{HL13}. Their proofs are straightforward
modifications of those in \cite{HL13}, and hence we omit them.

\begin{rem}\label{rem:subadditive}
The sequence $a_k(S)$ is sub-additive for every set $S$ in any normed (or
metric) space. More precisely, for any positive integers $k,l$, we have $a_{k+l}(S) \leq a_k(S) + a_l(S)$.
\end{rem}

\begin{rem}\label{rem:prelim}
Let $S$ be a set of diameter $d > 0$ in a normed (or metric) space. Then
for every set $S$ of diameter $d > 0$ and every $k \geq 1$, we have
$a_k(S) \geq 2k$. Furthermore, for every value of $k$, if $a(S) = 2$, then
$a_k(S) = 2k$, and
if $a(S) > 2$, then $a_k(S) > 2k$.
\end{rem}

\begin{rem}\label{rem:boundary}
Let $S \subset \Re^n$ be a set of positive diameter in a normed space.
Then for every value of $k$, $a_k(S) = a_k(\bd S)$.
\end{rem}

Let $S$ be a bounded set in a normed plane. By Theorem~\ref{thm:normedBorsuk},
if $S$ is not a set of unique completion then $a(S)=2$, which yields that for
any $k$,
$a_k(S) = 2k$.
Combined with Corollary~\ref{cor:udmBorsukcovering}, it yields that it suffices
to
characterize the $k$-fold Borsuk numbers of complete sets.
To do this, we need a generalization of the notion of Reuleaux polygons for
normed planes (cf. also \cite{T96}, \cite{S69} and \cite{H63}).

\begin{defn}\label{defn:Reuleaux}
Let $C$ be a complete set in a normed plane.
If $C$ is the intersection of finitely many translates of $\B$, we say that
$C$ is a \emph{Reuleaux polygon}. If $m$ is the smallest number such that $C$
is the
intersection of $m$ translates of $\B$,
then we say that \emph{$C$ has $m$ sides}.
\end{defn}


\begin{thm}\label{thm:notReuleaux}
Let $C$ be a complete set of diameter one in a normed plane, which is not a
Reuleaux
polygon.
Then for every $k$, $a_k(C) = 2k+1$.
\end{thm}

\begin{proof}
Clearly, by Remark~\ref{rem:prelim}, for every $k$, we have $a_k(C) \geq 2k+1$.
Thus, we need to show that if $C$ is not a Reuleaux polygon, then $C$, or
equivalently, $\bd C$, can be covered $k$-fold by $2k+1$ subsets
of smaller diameters.

To do this, we prove the existence of $2k+1$ diameters $[p_i,p_{2k+1+i}]$, where
$i=1,2,\ldots,4k+2$ of $C$, such that for any $j \neq 2k+1+i$,
$[p_i,p_j]$ is not a diameter of $C$.
Observe that from this, the assertion follows.
Indeed, by Lemma~\ref{lem:monotonicity}, we have that any two of these diameters intersect.
Thus, we may label their endpoints in such a way that
$p_1,p_2, \ldots, p_{4k+2}$ are in
counterclockwise order in $\bd C$.
Let $A_i$ be the arc of $\bd C$, connecting $p_i$ and $p_{2k+i}$ and not containing $p_{2k+1+i}$.
Then $A_i$ is of diameter less than one, and the arcs $A_{1+ks}$, where $s=1,2,\ldots,2k+1$, form a $k$-fold Borsuk covering of $\bd C$.

For simplicity, for any point $x \in \bd C$, we set $G(x) = C \cap \S(x)
\subset \bd C$.
We choose the required diameters as follows. Let $[p_1,p_{2k+2}]$ be an arbitrary
diameter of $C$.
Let $q_1, r_1$ and $q_{2k+2}, r_{2k+2}$ be the endpoints of the arcs $G(p_1)$ and $G(p_{2k+2})$, respectively. 
It follows from Lemma~\ref{lem:monotonicity} that $G(p_{2k+2}) \subseteq G(q_1) \cup G(r_1)$ and $G(p_1) \subseteq G(q_{2k+2}) \cup G(r_{2k+2})$.
Then, as no finitely many unit circle arcs cover $\bd C$, $X_2 = \bd C \setminus
(G(q_1) \cup G(r_1) \cup G(q_{2k+2}) \cup G(r_{2k+2})) \neq \emptyset$.

Observe that for any $x \in X_2$, $||x-p_1||_{\B}$ and $||x-p_{2k+2}||_{\B}$ are strictly
less than one, and any point diametrically opposite to $x$ is also contained in $X_2$.
Let $p_2 \in X_2$ arbitrary. Since $C$ is complete, there is some
$p_{2k+3} \in \bd C$ such that $[p_2,p_{2k+3}]$ is a diameter of $C$.
Then $p_{2k+3} \in X_2$; that is, $||p_{2k+3}-p_1||_{\B}$ and $||p_{2k+3}-p_{2k+2}||_{\B}$ are strictly less than one.
Let us define $q_2, r_2, q_{2k+3}, r_{2k+3}$ similarly as for $p_1$ and $p_{2k+2}$.
Now, set $X_3 = X_2 \setminus (G(q_2) \cup G(r_2) \cup G(q_{2k+3}) \cup G(r_{2k+3}))$.
Since $\bd C$ is not covered by finitely many unit circle arcs, we have $X_3 \neq \emptyset$.
Thus, using the argument as for $[p_2,p_{2k+3}]$, we can find a diameter $[p_3,p_{2k+4}]$
with $p_3,p_{2k+4} \in X_3$, satisfying the required conditions.
Since $C$ is not a Reuleaux polygon, repeating this procedure we may choose the required $2k+1$ diameters
for any value of $k$.
\end{proof}

\begin{thm}\label{thm:Reuleaux}
If $C$ is an $m$-sided Reuleaux polygon of diameter one in the not angled norm
with unit disk $\B$, then
\begin{enumerate}
\item $m$ is an odd integer,
\item if $m=2s+1$, then the $k$-fold Borsuk number of $C$ is $a_k(C) = 2k +
\left\lceil \frac{k}{s} \right\rceil$.
\end{enumerate}
\end{thm}

\begin{proof}
Let $G_i$, where $i=1,2,\ldots,m$, be unit circle arcs that cover $\bd C$, and
let $p_i, q_i$ and $r_i$  be the center and the two endpoints of $G_i$,
respectively. Clearly, we may assume that no $G_i$ is a proper subset of any
unit circle arc in $\bd C$.

We label the points in such a way that in counterclockwise order, $q_i$ is the
starting and $r_i$ is the endpoint of $G_i$, and
the points $q_1,q_2,\ldots,q_m$ are in this counterclockwise order in $\bd C$.
For simplicity, we call the $G_i$s the \emph{sides}, and their endpoints the
\emph{vertices} of $C$.
Note that $r_1,r_2,\ldots,r_m$ are in this counterclockwise order as well, as
otherwise $G_i \subset G_j$ for some $i \neq j$, which contradicts the
assumption
that $C$ is $m$-sided.
By Lemma~\ref{lem:monotonicity}, we have that $p_1,p_2,\ldots,p_m$ are also in
this counterclockwise
order.

Since $C$ is complete, $p_i \in \bd C$ for every value of $i$.
Furthermore, since $m$ is the minimal number of unit circle arcs that cover $\bd
C$,
there is no point that belongs to more than two arcs. We observe also that if
$p_i$ is in the relative interior of a segment $[x,y] \subset \bd C$, then, by
Lemma~\ref{lem:disjointdiam},
$G_i = [q_i,r_i]$ is a segment.
Thus, replacing $G_i$ by, say $\S(x) \cap C$, we still have a family of
$m$ unit circle arcs that cover $\bd C$.
This implies that, without loss of generality, we may assume that no $p_i$ is in
the relative interior of a segment on $\bd C$.

Consider, first, the case that two consecutive sides, say $G_i$ and $G_{i+1}$
overlap. Then $q_i,q_{i+1},r_i$ and $r_{i+1}$
are in this counterclockwise order in $\bd C$. Thus, $[p_i,r_i]$ and
$[p_{i+1},q_{i+1}]$ are disjoint diameters, which
yields, by Lemma~\ref{lem:disjointdiam}, that $[p_i,p_{i+1}], [q_{i+1},r_i]
\subset \bd C$, and that they are parallel.
Hence, for any two overlapping sides of $C$, the common part is a straight line
segment.

Now we show that the intersection of any two consecutive sides of $C$ contains
the center of exactly one side.
Consider the sides $G_i$ and $G_{i+1}$.

\emph{Case 1}, $G_i$ and $G_{i+1}$ do not overlap.
Then $r_i = q_{i+1}$.
Observe that $p_i,p_{i+1} \in \S(r_i) \cap C$.
Let $G$ be the arc of $\bd C$ connecting $p_i$ and $p_{i+1}$ and not
containing $r_i$.
We show that there is a point in the relative interior of $G$ which is
diametrically opposite only to $r_i$.
Note that since $C$ is a Reuleaux polygon, it yields that in this case $C \cap
\S(r_i)$ must be a side of $C$.

Let $p$ be an arbitrary relative interior point of $G$, and assume that $C \cap
\S(p)$ contains not only $r_i$, but some other point $x$ as well.
Without loss of generality, we may assume that $x \in G_i$, which yields that
$[p_i,r_i]$ and $[p,x]$ are disjoint diameters of $C$.
Thus, by Lemma~\ref{lem:disjointdiam}, $[p,p_i], [r_i,x] \subset \bd C$, and
they are parallel.
Since here $p$ is an arbitrary relative interior point of $G$, we have that
either $G=[p_i,p_{i+1}]$
or there is some relative interior point $z$ of $G$ such that $G=[p_i,z] \cup
[z,p_{i+1}]$.
Observe that $G \subset C \cap \S(r_i)$, and hence, as $\B$ is not angled,
it follows that $G=[p_i,p_{i+1}]$.
Furthermore, for some point $x \in C$, we have that $[r_i,x] \subset \bd C$, and
that $[r_i,x]$ and $[p_i,p_{i+1}]$ are parallel.
This means that $[r_i,x]$ belongs to both $\S(p_i)$ and $\S(p_{i+1})$,
which contradicts our assumption that $G_i$ and $G_{i+1}$
do not overlap.

\emph{Case 2}, $G_i$ and $G_{i+1}$ overlap; or in other words, $r_i \neq
q_{i+1}$.
Then, similarly like in Case 1, we have that $[r_i,q_{i+1}], [p_i,p_{i+1}]
\subset \bd C$, and they are parallel.
Let $L$ and $L'$ denote the line containing $[p_i,p_{i+1}]$ and $[r_i,q_{i+1}]$,
respectively.
Observe that $\S(p_i)$ and $\S(p_{i+1})$ both contain $C \cap L'$, and
thus,
we have $C \cap L' = [r_i,q_{i+1}]$.
Furthermore, note that, for any point $p$ in the relative interior of
$[p_i,p_{i+1}]$,
the points of $C$ diametrically opposite to $p$ are exactly the points of
$[r_i,q_{i+1}]$.
Thus, the center of any side of $C$ containing $p$ is a point of
$[q_{i+1},r_i]$.
Since we chose the sides of $C$ in such a way that no center is contained in a
straight line segment in $\bd C$,
we have that only $q_{i+1}$ or $r_i$ can be the center of a side, and also that
$L \cap C = [p_i,p_{i+1}]$.

Suppose, for contradiction, that both $q_{i+1}$ and $r_i$ are centers, and let
these sides be $G_j$ and $G_{j+1}$.
Then, we have $[p_i,p_{i+1}] \subseteq G_j \cap G_{j+1}$, and, similarly like in
the previous paragraph, we may obtain that $[p_i,p_{i+1}] = G_j \cap G_{j+1}$.
Thus, $q_{j+1} = p_i$ and $r_j = p_{i+1}$.
Since $q_i \neq q_{i+1}=p_j$ and $q_j \neq q_{j+1}=p_i$, it follows that
$[q_{i+1},q_j]$ and $[p_i,q_i]$ are disjoint diameters of $C$.
Hence, by Lemma~\ref{lem:disjointdiam}, we have that $[q_i,q_{i+1}]$ and
$[q_j,p_i]$ are parallel and contained in $\bd C$.
Thus, $[q_i,q_{i+1}]$ and $[q_{i+1},r_i]$ are both contained in $\S(p_i)$,
which contradicts our assumption that the normed plane is not angled.

We have shown that the intersection of any two consecutive sides contains the
center of exactly one side.
Since any point of $\bd C$ belongs to at most two sides of $C$, these
intersections are pairwise disjoint.
As the number of centers is equal to the number of intersections, it follows
that the center of every side of $C$
is contained in one of these intersections.
In fact, we showed a bit more: every center is the vertex of some other side.

For every value of $i$, consider a point $z_i$ that belongs to $G_i$ but no
other side of $C$.
Note that since no point of $\bd C$ belongs to more than two sides of $C$, this
is possible, and also that,
by Lemma~\ref{lem:disjointdiam}, the segments $[p_i,z_i]$, where
$i=1,2,\ldots,m$, are pairwise intersecting diameters of $C$.
Clearly, the $2m$ points $p_i$ and $z_j$ form an alternating sequence $S$ in
$\bd C$, and each of the two open arcs of $\bd C$,
starting at, say, $p_1$ and ending at $z_1$, contains exactly $m-1$ points.
Since the subsequence of $S$ in any of the above two arcs, starts with some
$z_i$ and ends with some $p_j$, we have that $m-1$ is an even number, and thus,
$m$ is odd.

Now we prove the second part.
Let $m=2s+1$. According to the previous paragraph, we have that for every $i$,
$p_i \in G_{i+s} \cap G_{i+s+1}$.
First, we show that the points $z_i$ can be chosen in such a way that the set
$Z=\{ z_i : i=1,2,\ldots,m\}$ contains
no diametrically opposite pair.

Assume that for every $i$, $z_i$ belongs to only $G_i$, but $Z$ contains a
diametrically opposite pair, say $z_i$ and $z_j$.
Then $j=i-s$ or $j=i+s$. Without loss of generality, we may assume that $z_i$
and $z_{i+s}$ are diametrically opposite.
From this, by Lemma~\ref{lem:disjointdiam}, we obtain that $[z_i,p_{i+s}]$ and
$[z_{i+s},p_i]$ are parallel and are contained in $\bd C$.
Let $L$ be the line containing $[p_{i+s},z_i]$.
Note that $p_{i+s}$ is an endpoint of $L \cap C$, and let $x$ be the other
endpoint.
Observe that $p_{i-s} \notin L$, as otherwise $z_{i+s} \in G_{i-s}$, which is a
contradiction.
In addition, $x$ is not diametrically opposite to $z_{i-s}$.
Indeed, if $[x,z_{i-s}]$ is a diameter, then $[x,z_{i-s}]$ and $[p_{i-s},p_i]$ are disjoint diameters, and thus
Lemma~\ref{lem:disjointdiam} yields that $[x,p_{i-s}],[z_{i-s},p_i] \subset \bd C$, which contradicts our assumption that
the normed plane is not angled (cf. Figure~\ref{fig:Reuleaux}).
Now we choose any point $y\in(\bd C)\setminus L$ sufficiently close to $x$, and replace $z_i$ by $y$.
Then, clearly, $y$ is diametrically opposite to neither $z_{i-s}$ nor $z_{i+s}$.
Thus, to choose a subset $Z$ that does not contain diametrically opposite
points, we start with any set and then, applying the argument of this paragraph, we
may replace the points one by one to reduce the number of diametrically opposite
points.

\begin{figure}[ht]
\begin{center}
\includegraphics[width=0.5\textwidth]{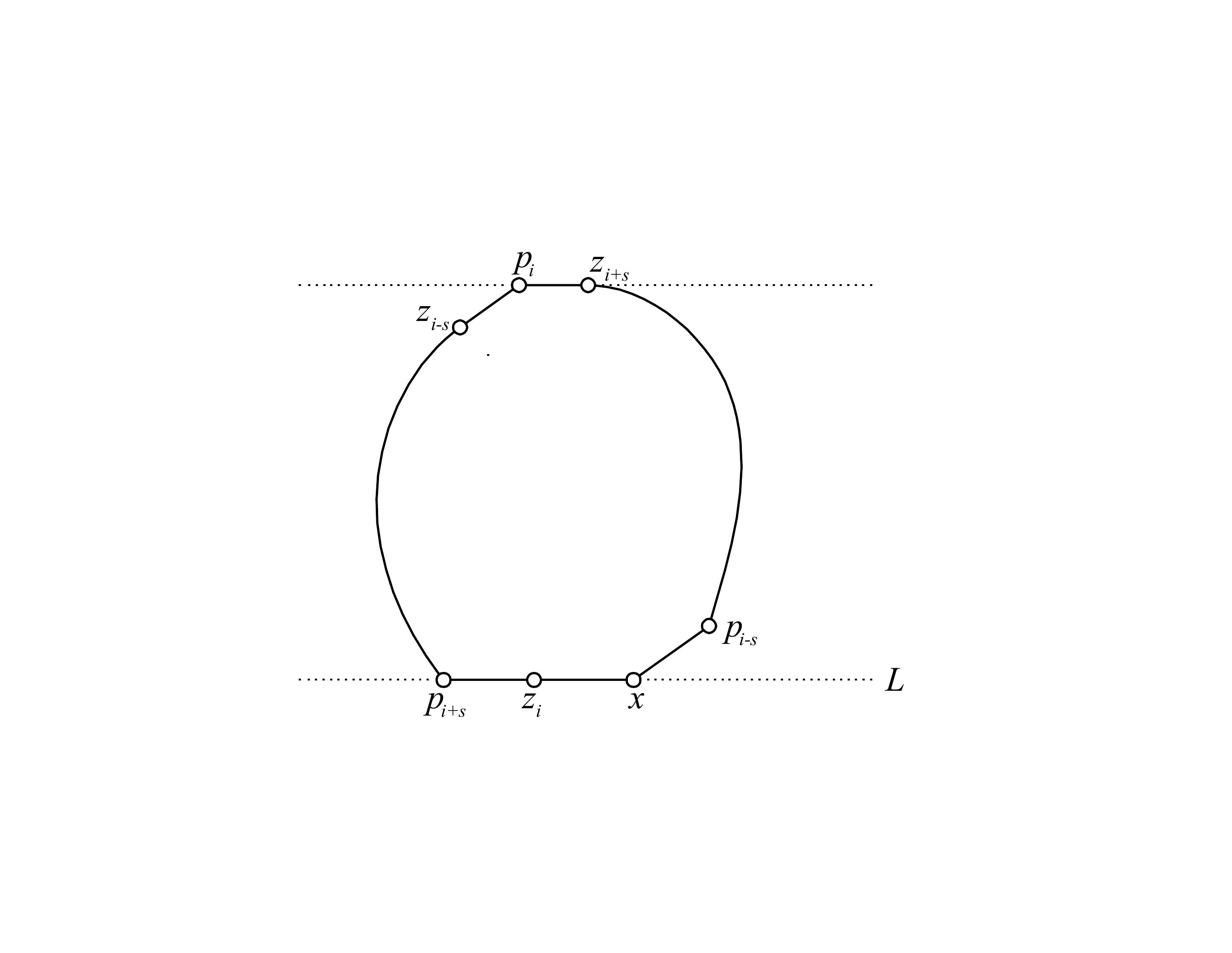}
\caption[]{An illustration for the proof of Theorem~\ref{thm:Reuleaux}}
\label{fig:Reuleaux}
\end{center}
\end{figure}

We constructed a subset $Z=\{ z_i: i=1,2,m\}$ such that for every $i$, $z_i$
belongs only to $G_i$, and $Z$ contains no diametrically opposite pair.
Let $A_i$ denote the closed arc of $\bd C$, which, in counterclockwise order,
starts at $z_i$ and ends at $z_{i+s}$.
Observe that by Lemma~\ref{lem:monotonicity} and the choice of $Z$, no such arc
contains a diametrically opposite pair.
On the other hand, the sets $A_{js}$, where $j=0,1,\ldots,2k + \left\lceil
\frac{k}{s} \right\rceil -1$ and the indices are taken $\mod m$,
covers $\bd C$ $k$-fold, and thus, they are a $k$-fold Borsuk covering of $\bd
C$. This proves that $a_k(C) \leq 2k + \left\lceil \frac{k}{s} \right\rceil$.

To prove the other direction, we note that the $k$-fold Borsuk coverings of the
set $\{p_i:i=1,2,\ldots,m\}$ can be identified with the $k$-fold
vertex-colorings of a $(2s+1)$-cycle. Since it is known (cf. \cite{S76}) that
the $k$-fold chromatic number of such a cycle is $2k + \left\lceil \frac{k}{s}
\right\rceil$,
the assertion follows.
\end{proof}

From Theorems~\ref{thm:normedBorsuk}, \ref{thm:notReuleaux} and
Remark~\ref{rem:prelim}, we immediately
obtain the following.

\begin{thm}\label{thm:angled}
Let $S$ be a set of positive diameter in a normed plane $\B$.
\begin{itemize}
\item If $S$ is not a set of unique completion, or $S$ does not satisfy the
supporting line
property, then for every value of $k$, $a_k(S) = 2k$.
\item If $S$ is a set of unique completion that satisfies the supporting line
property (see page~\pageref{supplineprop}) and the
completion of $S$ is not a Reuleaux polygon, then
for every $k$, $a_k(S) = 2k+1$.
\end{itemize}
\end{thm}

\section{Remarks and Questions}

\begin{rem}
We note that our results cannot be extended to angled planes. For example,
Theorem~\ref{thm:completionNotAngled} fails if the unit disk $\B$ is a parallelogram. Besides, any
centrally symmetric polygon with $4m$ sides is a Reuleaux polygon with $2m$
sides
in its norm (and thus, it has even sides according to our definition).
\end{rem}

\begin{rem}
The $k$-fold Borsuk number of an $o$-symmetric polygon $P$ with $2m$ sides in its own norm is
$a_k(P)=2k+\left\lceil \frac{2k}{m-1} \right\rceil$.
\end{rem}

\begin{proof}
Let the vertices of the polygon be $p_1,p_2,\ldots,p_{2m}$ in counterclockwise order. Then $p_i$ is diametrically opposite
to $p_{i+k-1}, p_{i+k}$ and $p_{i+k+1}$. Thus, the inequality $a_k(P) \geq 2k+\left\lceil \frac{2k}{m-1} \right\rceil$ follows
from the Pigeonhole Principle. On the other hand, if $G_i$ denotes the shorter arc in $\bd P$, connecting the midpoints of $[p_i,p_{i+1}]$
and $[p_{a+k-1},p_{i+k}]$, then, clearly, $G_i$ contains the vertices of no diameter of $P$.
Thus, the arcs $G_{i+t(k-1)}$, where $t=1,2,\ldots,2k+\left\lceil \frac{2k}{m-1} \right\rceil$, form a $k$-fold Borsuk-covering of $\bd P$.
\end{proof}

\begin{rem}
It is proven in \cite{BS77} that in any angled normed plane there is a complete
set of Borsuk number two. In other words, for a normed plane, the result in
\cite{B70} about the Borsuk numbers of sets in the Euclidean plane
holds in the same form if, and only if the plane is not angled. According to our results, the same can be observed
about the multiple Borsuk numbers of sets.
\end{rem}

\begin{rem}
In any angled normed plane, there is a Borsuk covering of a set of unique
completion, satisfying the supporting line property (see page~\pageref{supplineprop}),
that cannot be completed to a Borsuk covering of its completion.
\end{rem}

\begin{proof}
If the norm is a parallelogram norm, the remark trivially follows.
Hence, we may assume that the unit disk $\B$ is not a parallelogram, and that its boundary contains $[x,y] \cup [y,z]$ and $[-x,-y] \cup [-y,-z]$.
Without loss of generality, we may assume that the lines, containing $[x,y]$ and $[y,z]$,
intersect $\B$ in $[x,y]$ and $[y,z]$, respectively.

Let $C$ be the truncation of $\B$ with a line connecting the relative interior points $w_1$ and $w_2$ of $[x,y]$ and $[y,z]$, respectively.
Clearly, the unique completion of $C$ is $\B$, and $C$ satisfies the supporting line property.
Let $w$ be the midpoint of $[w_1,w_2]$.
Let $u_1$ and $u_2$ be relative interior points of $[-x,-y]$ and $[-y,-z]$, respectively (cf. Figure~\ref{fig:example}).
Then the shorter arcs of $\bd C$ connecting $w$ to $u_1$, $u_1$ to $u_2$, and $u_2$ to $w$, is a Borsuk covering of $\bd C$.
On the other hand, $y$ cannot be added to any of these arcs, which yields that this covering cannot be extended to any
Borsuk covering of $\B$.
\end{proof}

\begin{figure}[ht]
\begin{center}
\includegraphics[width=0.5\textwidth]{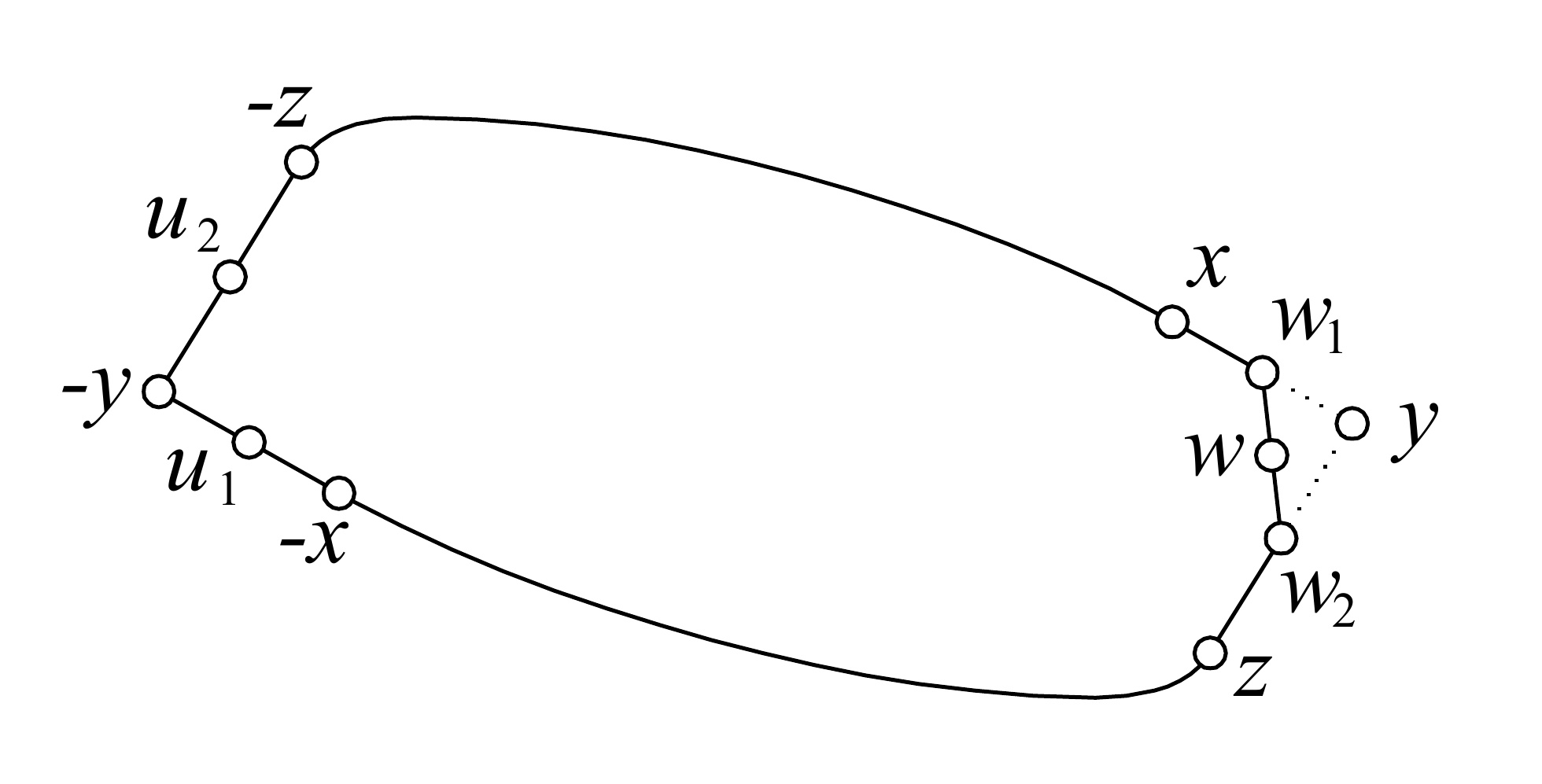}
\caption[]{A Borsuk covering may not be extended in an angled plane}
\label{fig:example}
\end{center}
\end{figure}

Note that if $\B$ is a parallelogram, then the only complete sets of unit
diameter are the translates of $\B$ (cf. \cite{Yo91}, \cite{NV04} or \cite{JL07}).

\begin{rem}
Let $S$ be a compact set with $a(S)=3$ in the normed plane where
$\B$ is a parallelogram. Then $a_k(S)= 3k$ for every $k$.
\end{rem}

\begin{proof}
Without loss of generality, let $\B$ be the unique completion of $S$.
By the supporting line property, $S$ contains at least two consecutive vertices of
$\B$.
Furthermore, since $\B$ is the unique completion, $S$ contains a point of the
opposite side of $\B$.
Thus, $S$ contains three points at pairwise normed distances equal to $\diam S$,
which yields
$a_k(S) \geq 3k$. By sub-additivity, we have $a_k(S) \leq 3k$, and the assertion
readily follows.
\end{proof}


\end{document}